\documentclass[12pt]{amsart}
\usepackage{amscd}
\usepackage{amsmath}
\usepackage{amssymb}
\usepackage{amsthm}
\usepackage{url}
\newtheorem{theorem}{Theorem}[section]
\newtheorem{lemma}[theorem]{Lemma}
\newtheorem{proposition}[theorem]{Proposition}
\newtheorem{corollary}[theorem]{Corollary}

\theoremstyle{definition}
\newtheorem{definition}[theorem]{Definition}
\newtheorem{example}[theorem]{Example}

\newtheorem*{acknowledgement}{Acknowledgment}
\theoremstyle{remark}
\newtheorem{remark}[theorem]{Remark}
\numberwithin{equation}{section}

\DeclareMathOperator{\conv}{CONV}

\def\Ac{{\mathcal A}}
\def\Bc{{\mathcal B}}
\def\Hc{{\mathcal H}}

\def\Fc{{\mathcal F}}
\def\Pc{{\mathcal P}}
\def\Qc{{\mathcal Q}}

\def\NZQ{\mathbb}               

\def\ZZ{{\NZQ Z}}
\def\RR{{\NZQ R}}

\def\hugesymbol#1{\mbox{\strut\rlap{\smash{\Huge$#1$}}\quad}}
\begin{document}
\title{The $\delta$-vectors of reflexive polytopes and of the dual polytopes}
\author{Akiyoshi Tsuchiya}
\address[Akiyoshi Tsuchiya]{Department of Pure and Applied Mathematics,
	Graduate School of Information Science and Technology,
	Osaka University,
	Suita, Osaka 565-0871, Japan}
\email{a-tsuchiya@cr.math.sci.osaka-u.ac.jp}
\begin{abstract}
	Let $\delta(\Pc)$ be the $\delta$-vector of a reflexive polytope $\Pc \subset \RR^d$ of dimension $d$ and $\delta(\Pc^\vee)$  the $\delta$-vector of the dual polytope $\Pc^\vee \subset \RR^d$.
	In general, $\delta(\Pc)=\delta(\Pc^\vee)$ does not hold.
	In this paper, we give a higher-dimensional construction of a reflexive polytope whose $\delta$-vector equals the $\delta$-vector of the dual polytope.
	In particular, we consider the case that the reflexive polytope and the dual polytope are unimodularly equivalent.
\end{abstract}
\maketitle
\section*{INTRODUCTION}
Let $\Pc \subset \RR^d$ be an integral convex polytope, that is, a convex polytope whose vertices have integer coordinates, of dimension $d$.
Given integers $n=1,2,\ldots,$ we write $i(\Pc,n)$ 
for the number of integer points belonging to $n\Pc$, 
where $n\Pc=\left\{n\alpha : \alpha \in \Pc \right\}$.
In other words,
$$i(\Pc,n)=|n\Pc \cap  \ZZ^d|, \ \ \ n=1,2,\ldots .$$
In the late 1950s, Ehrhart succeeded in proving that $i(\Pc, n)$ 
is a polynomial in $n$ of degree $d$ with $i(\Pc,0)=1$.
(\cite[Example 35.11]{Hibi}).
		We call $i(\Pc,n)$ the \textit{Ehrhart \ polynomial} of $\Pc$.
	
	The generating function of the integral point enumerator, i.e., the formal power series
	$$\text{Ehr}_\Pc(t)=1+\sum\limits_{n=1}^{\infty}i(\Pc,n)t^n$$
	is called the \textit{Ehrhart series} of $\Pc$.
	It is well known that it can be expressed as a rational function of the form
	$$\text{Ehr}_\Pc(t)=\frac{\delta_0+\delta_1t+\cdots + \delta_dt^d}{(1-t)^{d+1}}.$$
	The sequence of the coefficients of the polynomial in the numerator 
	$$\delta(\Pc)=(\delta_0,\delta_1,\ldots,\delta_d)$$
	is called the \textit{$\delta$-vector} of $\Pc$.
	
	The $\delta$-vector has the following properties:
	\begin{itemize}
		\item$\delta_0=1,\delta_1=|\Pc \cap \ZZ^d|-(d+1)$ and $\delta_d=|(\Pc\setminus \partial \Pc) \cap \ZZ^d|$. Hence, $\delta_1 \geq \delta_d$; 
		\item Each $\delta_i$ is nonnegative (\cite{RS_DRCP});
		\item If $\delta_d \neq 0$, then one has $\delta_1 \leq \delta_i$ for every $1 \leq i \leq d-1$ (\cite{H_LBTEP}).
	\end{itemize}
	We refer the reader to \cite{Beck,Ehrhart,Hibi,RS_ECV1,RS_OHGCM,RS_MPH} for further informations on Ehrhart 
	polynomials and $\delta$-vectors.
	
	An integral  convex polytope is called \textit{reflexive} if the origin 
	of $\RR^d$ is a unique integer point belonging to the 
	interior $\Pc-\partial \Pc$ of $\Pc$ 
	and its dual 
	$$\Pc^\vee:=\left\{y \in \RR^d :\langle x,y \rangle \leq 1 \ \text{for all} \ x\in \Pc \right\}$$
	is also an integral polytope,
	where $\langle x,y \rangle$ is the 
	usual inner product of $\RR^d$.
	
	Let $\Pc \subset \RR^d$ be an integral convex polytope of dimension $d$ containing the origin in its interior 
	and $\delta(\Pc)=(\delta_0,\delta_1,\ldots,\delta_d)$ its $\delta$-vector.
	It follows from \cite{Hibi_DPRCP} that the following 
	conditions are equivalent:
	\begin{itemize}
		\item $\Pc$ is reflexive;
		\item $\delta(\Pc)$ is symmetric, i.e.,
		$\delta_i=\delta_{d-i}$ for every $0 \leq i \leq d$.
	\end{itemize}
	Let $\Pc \subset \RR^d$ be a reflexive polytope of dimension $d$.
	In general, $\delta(\Pc)=\delta(\Pc^\vee)$ does not hold.

It is known from work of Lagarias and Ziegler (\cite{Lag}) that there are only finitely many reflexive polytopes (up to unimodular equivalence) in each dimension, with one reflexive polytope in dimension one, $16$ in dimension two, $4,319$ in dimension three, and $473,800,776$ in dimension four according to computations by Kreuzer and Skarke (\cite{Kre}).
By computing the number of reflexive polytopes whose $\delta$-vectors equal the $\delta$-vectors of the dual polytopes, we find that there are $4$ such reflexive polytopes in dimension two and $327$ in dimension three.

It is known that for each $d \geq 2$ there exists a reflexive simplex  of dimension $d$ whose $\delta$-vector equals the $\delta$-vector of the dual polytope (\cite{Nill}).
In particular, this reflexive simplex and the dual polytope are unimodularly equivalent.
However,  other examples of such a reflexive polytope are known very little.

In this paper, we give examples of a reflexive polytope whose $\delta$-vector equals the $\delta$-vector of the dual polytope.
In section 1, we give a higher-dimensional construction of a reflexive polytope whose $\delta$-vector equals the $\delta$-vector of the dual polytope (Theorem \ref{main}).
In section 2, we give a new example of a reflexive simplex whose $\delta$-vector equals the $\delta$-vector of the dual polytope (Theorem \ref{sim}).

\section{a higher-dimensional construction of a special reflexive polytope}

Let $\ZZ^{d \times d}$ denote the set of $d \times d$ integral matrices.
Recall that a matrix $A \in \ZZ^{d \times d}$ is {\em unimodular} if $\det (A) = \pm 1$.
Given integral convex polytopes $\Pc$ and $\Qc$ in $\RR^d$ of dimension $d$,
we say that $\Pc$ and $\Qc$ are {\em unimodularly equivalent}
if there exists a unimodular matrix $U \in \ZZ^{d \times d}$
and an integral vector $w$, such that $\Qc=f_U(\Pc)+w$,
where $f_U$ is the linear transformation in $\RR^d$ defined by $U$,
i.e., $f_U({\bf v}) = {\bf v} U$ for all ${\bf v} \in \RR^d$.
In this case we write $\Pc \cong \Qc$.
Clearly, if $\Pc \cong \Qc$, then
$\delta(\Pc) = \delta(\Qc)$.
Moreover, if $\Pc \subset \RR^d$, then $\text{Vol}(\Pc)$ denotes the (normalized) \textit{volume} of $\Pc$, i.e., $d!$ times the usual euclidean volume of $\Pc$.

\begin{example}
	\label{eq_2d}
	Let $\Pc \subset \RR^2$ be the reflexive polytope with the vertices $(1,0)$, $(-1,2)$ and $(-1,-1)$.
	Then the dual polytope $\Pc^\vee$ has the vertices $(-1,0)$, $(1,-2)$ and $(1,1)$.
	Clearly, $\Pc \cong \Pc^\vee$.
	Therefore $\delta(\Pc)=\delta(\Pc^\vee)$.
\end{example}

However,  other examples of such a reflexive polytope, i.e., a reflexive polytope and the dual polytope are unimodularly equivalent, are known very little.
In this section, we give a higher-dimensional construction of a reflexive polytope whose $\delta$-vector equals the $\delta$-vector of the dual polytope.

For $d \geq 2$ and  an integral convex polytope $\Pc \subset \RR^{d-1}$ of 
dimension $d-1$, we set
\begin{displaymath}
\begin{aligned}
\Ac(\Pc)&=\Pc \times [-1,1],\\
\Bc(\Pc)&=\conv (\left\{\Pc \times \{0\},(0,0,\ldots,0,1),(0,0,\ldots,0,-1)\right\}) \subset \RR^{d},\\
\Gamma(\Pc)&=\conv (\left\{\Pc \times [-1,0],(0,0,\ldots,0,1)\right\}) \subset \RR^{d}.
\end{aligned}
\end{displaymath}
We recall that if $\Pc$ is reflexive, then $\Ac(\Pc)$ and $\Bc(\Pc)$ are also reflexive.
Moreover, we have $\Ac(\Pc)^\vee =\Bc(\Pc^\vee)$ and $\Bc(\Pc)^\vee = \Ac(\Pc^\vee)$.
$\Gamma(\Pc)$ is an analogy between $\Ac(\Pc)$ and $\Bc(\Pc)$.

At first, we show that if $\Pc$ is reflexive, then $\Gamma(\Pc)$ is a reflexive polytope of dimension $d$.

\begin{proposition}
	\label{refl}
	For $d \geq 2$, let $\Pc \subset \RR^{d-1}$ be a reflexive polytope of 
	dimension $d-1$.
	Then $\Gamma(\Pc)$ is a reflexive polytope of dimension $d$.
	Moreover, $\Gamma(\Pc^\vee) \cong \Gamma(\Pc)^\vee$.
\end{proposition}

Before proving Proposition \ref{refl}, we give the following lemma.

\begin{lemma}[{\cite[Corollary 35.6]{Hibi}}]
	\label{facet}
	Let $\Pc \subset \RR^d$ be an integral convex polytope 
	dimension $d$ containing the origin in its interior. 
	Then a point $a \in \RR^d$ is a vertex of $\Pc^\vee$ if 
	and only if $\Hc \cap \Pc$ is a facet of $\Pc$,
	where $\Hc$ is the hyperplane
	$$\left\{ x \in \RR^d \ | \ \langle a, x \rangle =1 \right\}$$
	in $\RR^d$.
\end{lemma}
Now, we prove Proposition \ref{refl}.

\begin{proof}[Proof of Proposition \ref{refl}]
	Let $\Fc_1,\ldots,\Fc_s$ be facets of $\Pc$ and for 
	$1 \leq i \leq s$,
	and let $\Hc_i$ be the hyperplane 
	satisfying $\Fc_i=\Pc \cap \Hc_i$.
	Then $\Gamma(\Pc)$ has $2s+1$ facets.
	By Lemma \ref{facet}, 
	we can assume that for $1 \leq i \leq s$,
	$$\Hc_i=\left\{ x \in \RR^{d-1} \ | \ \langle a_i,x \rangle =1\right\},$$
	where $a_i \in \ZZ^{d-1}$.
	Set 
	\begin{displaymath}
	\Fc _i'=\left\{        
	\begin{aligned} 
	&\conv(\left\{\Fc_i \times \left\{ 0 \right\},(0,\ldots, 0,1)\right\}) &i&=1,\ldots ,s, \\
	&\Fc _{i-s} \times  [-1,0] &i&=s+1,\ldots ,2s, \\
	&\Pc \times \left\{ -1 \right\} &i&=2s+1.
	\end{aligned}
	\right.
	\end{displaymath}
	Then $\Fc'_1,\ldots,\Fc_{2s+1}'$ are facets of $\Gamma(\Pc)$.
	For $1 \leq i \leq 2s+1$ let $\Hc_i'$ be the hyperplane 
	satisfying $\Fc_i'=\Gamma(\Pc) \cap \Hc_i'$.
	Then 
	\begin{displaymath}
	\Hc_i'=\left\{        
	\begin{aligned} 
	&\left\{x \in \RR ^{d} \ | \ \langle (a_i,1),x \rangle =1 \right\} &i&=1,\ldots ,s,\\
	&\left\{x \in \RR ^{d} \ | \ \langle (a_{i-s},0),x\rangle =1 \right\} &i&=s+1,\ldots ,2s,\\
	&\left\{x \in \RR ^{d} \ | \ \langle (0,\ldots ,0,-1),x \rangle =1 \right\} &i&=2s+1.
	\end{aligned}
	\right.
	\end{displaymath}
	Hence by Lemma \ref{facet}, $\Gamma(\Pc)$ is a reflexive polytope of dimension $d$.
	
	Moreover, since $a_1,\ldots,a_s$ are the vertices of $\Pc^\vee$, it clearly follows that  $\Gamma(\Pc^\vee) \cong \Gamma(\Pc)^\vee$.
\end{proof}

Next, we present a direct formula for the computation of the $\delta$-vector of $\Gamma(\Pc)$ in terms of the $\delta$-vector of $\Pc$.

\begin{proposition}
	\label{maindelta}
	For $d \geq 2$, let $\Pc \subset \RR^{d-1}$ be an integral convex polytope of 
	dimension $d-1$, and
	we let $\delta(\Pc)=(\delta_0(\Pc),\delta_1(\Pc),\dots,\delta_{d-1}(\Pc))$ and
	$\delta(\Gamma(\Pc))=(\delta_0(\Gamma(\Pc)),\delta_1(\Gamma(\Pc)),\dots,\delta_{d}(\Gamma(\Pc)))$ be the $\delta$-vectors of $\Pc$ and $\Gamma(\Pc)$.
	Then for $i=0,1,\ldots, d$, we have 
	$$\delta_i(\Gamma(\Pc))=(i+1)\delta_i(\Pc)+(d-i+1)\delta_{i-1}(\Pc),$$
	where $\delta_{-1}(\Pc)=\delta_{d}(\Pc)=0$.	
\end{proposition}

In order to prove Proposition \ref{maindelta}, we use the following lemmas.

\begin{lemma}
	\label{prismdelta}
	For $d \geq 2$, let $\Pc \subset \RR^{d-1}$ be an integral convex polytope of 
	dimension $d-1$.
	Set $$\Qc=\Pc \times [0,1] \subset \RR^{d},$$
	and we let $\delta(\Pc)=(\delta_0(\Pc),\delta_1(\Pc),\dots,\delta_{d-1}(\Pc))$ and
	$\delta(\Qc)=(\delta_0(\Qc),\delta_1(\Qc),\dots,\delta_{d}(\Qc))$ be the $\delta$-vectors of $\Pc$ and $\Qc$.
	Then for $i=0,1,\ldots, d$, we have 
	$$\delta_i(\Qc)=(i+1)\delta_i(\Pc)+(d-i)\delta_{i-1}(\Pc),$$
	where $\delta_{-1}(\Pc)=\delta_{d}(\Pc)=0$.
\end{lemma}
\begin{proof}
We know $i(\Qc,n)=(n+1) \cdot i(\Pc,n)$.
Hence we have
 $$\delta_j(\Qc)=\sum\limits_{k=0}^{j}\binom{d+1}{k}(-1)^k(j-k+1) \cdot i(\Pc,j-k).$$
 Since
  $$\delta_j(\Pc)=\sum\limits_{k=0}^{j}\binom{d}{k}(-1)^k \cdot i(\Pc,j-k),$$
  we obtain
\begin{displaymath}
\begin{aligned}
(j+1)\delta_j(\Pc)+(d-j)\delta_{j-1}(\Pc)
&=\delta_j(\Qc),
\end{aligned}
\end{displaymath}  
  as desired.
\end{proof}

\begin{lemma}[{\cite[Theorem 2.4]{Beck}}]
	\label{delta}
		For $d \geq 2$, let $\Pc \subset \RR^{d-1}$ be an integral convex polytope of 
		dimension $d-1$.
		Set 
		$$\Qc=\conv (\left\{\Pc \times \{0\},(0,0,\ldots,0,1)\right\}) \subset \RR^{d}.$$
       Then we have $\textnormal{Ehr}_\Qc(t)=\textnormal{Ehr}_\Pc(t)/(1-t)$.
\end{lemma}

Now, we prove Propsition \ref{maindelta}.
\begin{proof}[Proof of Proposition \ref{maindelta}]
We set
	$\Qc_1=\conv ( \left\{\Pc \times \left\{ 0 \right\},(0,\ldots ,0,1)\right\})$
	and
	$\Qc_2=\Pc \times [-1,0]$.
	Then 
	$\Qc_1 \cup \Qc_2=\Gamma(\Pc)$
	and
	$\Qc_1 \cap \Qc_2=\Pc \times \left\{0\right\}$.
	Hence we have
	\begin{displaymath}
	\begin{aligned}
	\text{Ehr}_{\Gamma(\Pc)}(t) 
							&=\text{Ehr}_{\Qc_1}(t)+\text{Ehr}_{\Qc_2}(t)-\text{Ehr}_\Pc(t).
	\end{aligned}
	\end{displaymath}
	By Lemma \ref{delta},  $(1-t) \cdot \text{Ehr}_{\Qc_1}(t)=\text{Ehr}_\Pc(t)$.
	Hence we have 
	\begin{displaymath}
	\begin{aligned}
	\text{Ehr}_{\Gamma(\Pc)}(t)&=t \cdot \text{Ehr}_{\Qc_1}(t)+\text{Ehr}_{\Qc_2}(t).
	\end{aligned}
	\end{displaymath}
	Let $\delta(\Qc_1)=(\delta_0(\Qc_1),\delta_1(\Qc_1),\dots,\delta_{d}(\Qc_1))$ and
	$\delta(\Qc_2)=(\delta_0(\Qc_2),\delta_1(\Qc_2),\dots,\delta_{d}(\Qc_2))$ be the $\delta$-vectors of $\Qc_1$ and $\Qc_2$.
	By Lemma \ref{prismdelta} and Lemma \ref{delta}, we have 
	\begin{displaymath}
	 \delta_i(\Qc_1)=\left\{
	\begin{aligned}
	&\delta_i(\Pc) &i&=0,\ldots,d-1,\\
	&0&i&=d,
	\end{aligned}
	\right.
	\end{displaymath}
	and for $i=0,\ldots,d$, we have
	$$\delta_i(\Qc_2)=(i+1)\delta_i(\Pc)+(d-i)\delta_{i-1}(\Pc),$$
	where $\delta_{-1}(\Pc)=\delta_{d}(\Pc)=1$.
	Hence for $i=0,\ldots,d$, 
	\begin{displaymath}
	\begin{aligned}
	\delta_i(\Gamma(\Pc))&=\delta_{i-1}(\Qc_1)+\delta_i(\Qc_2)\\
						&=\delta_{i-1}(\Pc)+(i+1)\delta_i(\Pc)+(d-i)\delta_{i-1}(\Pc)\\
						&=(i+1)\delta_i(\Pc)+(d-i+1)\delta_{i-1}(\Pc),
	\end{aligned}
	\end{displaymath}
	as desired.
\end{proof}

We give a higher-dimensional construction of a reflexive polytope whose $\delta$-vector equals the $\delta$-vector of the dual polytope.
By following theorem, we obtain the construction.

\begin{theorem}
	\label{main}
For $d \geq 2$, let $\Pc$ and $\Qc \subset \RR^{d-1}$ be integral convex polytopes of 
dimension $d-1$ such that each of them has at least one interior integer point.
Then we have the following properties:
\begin{itemize}
	\item[(a)] $\Pc \cong \Qc$ if and only if $\Gamma(\Pc) \cong \Gamma(\Qc)$;
\item[(b)] $\delta(\Pc)=\delta(\Qc)$ if and only if $\delta(\Gamma(\Pc))=\delta(\Gamma(\Qc))$,
\end{itemize}		
where $\delta(\Pc)$, $\delta(\Qc)$, $\delta(\Gamma(\Pc))$ and $\delta(\Gamma(\Qc))$ are the $\delta$-vectors of $\Pc$, $\Qc$, $\Gamma(\Pc)$ and $\Gamma(\Qc)$.
\end{theorem}
\begin{remark}
	For $d \geq 2$, let $\Pc \subset \RR^{d-1}$ be a reflexive polytope of 
	dimension $d-1$.
	Then by Proposition \ref{refl} and Theorem \ref{main}, we have the following properties:
	\begin{itemize}
		\item[(a)] $\Pc \cong \Pc^\vee$ if and only if $\Gamma(\Pc) \cong \Gamma(\Pc)^\vee$;
		\item[(b)] $\delta(\Pc)=\delta(\Pc^\vee)$ if and only if $\delta(\Gamma(\Pc))=\delta(\Gamma(\Pc)^\vee)$,
	\end{itemize}		
	where $\delta(\Pc)$, $\delta(\Pc^\vee)$, $\delta(\Gamma(\Pc))$ and $\delta(\Gamma(\Pc)^\vee)$ are the $\delta$-vectors of $\Pc$, $\Qc$, $\Gamma(\Pc)$ and $\Gamma(\Pc)^\vee$.
\end{remark}
Now, we prove Theorem \ref{main}.
\begin{proof}[Proof of Theorem \ref{main}]
	(a) Clearly,  if $\Pc \cong \Qc$, then $\Gamma(\Pc) \cong \Gamma(\Qc)$.
	Conversely, suppose that  $\Gamma(\Pc) \cong \Gamma(\Qc)$.
	We can assume that  the origin of $\RR^{d-1}$ belongs to the interior of $\Pc$ and the interior of $\Qc$, and there exists a unimodular matrix $U \in \ZZ^{d \times d}$
	such that $\Gamma(\Qc)=f_U(\Gamma(\Pc))$,
	where $f_U$ is the linear transformation in $\RR^{d}$ defined by $U$.
	Let $v_1,\ldots,v_s$ be the vertices of $\Pc$ and $w_1,\ldots,w_s$ be the vertices of $\Qc$,
	and let
	\begin{displaymath}
	U=\left(  
	\begin{array}{cccc}
		u_{11}  & u_{12}  & \cdots  & u_{1d} \\
		u_{21} & u_{22} & \cdots  & u_{2d} \\
		\vdots & \vdots  &  \ & \vdots \\
		u_{d1} & u_{d2} & \cdots & u_{dd}
	\end{array}
	\right ).
\end{displaymath}
	Since$(0,\ldots,0,1)$ and $(0,\ldots,0,-1)$ belong to $\Gamma(\Pc)$
    and since for each $(x_1,\ldots,x_d) \in \Gamma(\Qc)$, we have $-1 \leq x_d \leq 1$,
	we know $-1 \leq u_{dd} \leq 1$.
	If $u_{dd}=0$, then  $f_U((v_i,0))$ and $f_U((v_i,-1))$ have a common $d$-th coordinate for $1 \leq i \leq s$.
	Since $\Gamma(\Qc)$ has just one vertex whose $d$-th coordinate equals $1$
	and since the $d$-th coordinate of $f_U((0,\ldots,0,1))$ equals $0$,
	there does not exist a vertex $v$ of $\Gamma(\Pc)$
	such that $f_U(v)=(0,\ldots,0,1)$,
	 a contradiction.
	If $v_{dd}=-1$, then $f_U((0,\ldots,0,1))=(0,\ldots,0,-1)$.
	However, $(0,\ldots,0,1)$ is a vertex of $\Gamma(\Pc)$ but $(0,\ldots,0,-1)$ is not a vertex of $\Gamma(\Qc)$, a contradiction.
	Hence $u_{dd}=1$.
	Since $f_U((0,\ldots,0,1))=(0,\ldots,0,1)$, we have
\begin{displaymath}
	U=\left(  
	\begin{array}{cccc}
	\  & \  & \  & \ast \\
	\ & \hugesymbol{U'}  & \  & \vdots \\
	\ & \  &  \ & \ast \\
	0 & \cdots & 0 & 1
	\end{array}
	\right ),
	\end{displaymath}
	where $U' \in \ZZ^{(d-1) \times (d-1)}$ is a unimodular matrix.  
    Then for each $v_i$ there exists a vertex $w_{j_i}$ of $\Qc$ such that $f_U((v_i,0))=(w_{j_i},-1)$
    and $f_U((v_i,1))=(w_{j_i},0)$.
    Hence for each $v_i$ we have $f_{U'}(v_i)=w_{j_i}$,
    where $f_{U'}$ is the linear transformation in $\RR^{d-1}$ defined by $U'$.
	Therefore, $\Pc \cong \Qc$.
	
	(b)	If $\delta(\Pc)=\delta(\Qc)$, by Proposition \ref{maindelta}, we have $\delta(\Gamma(\Pc))=\delta(\Gamma(\Qc))$.
	Suppose that $\delta(\Gamma(\Pc))=\delta(\Gamma(\Qc))$.
	We set
	$$\delta(\Pc) =(\delta_0(\Pc),\delta_1(\Pc),\dots,\delta_{d-1}(\Pc)),$$
	$$\delta(\Qc)=(\delta_0(\Qc),\delta_1(\Qc),\dots,\delta_{d-1}(\Qc)).$$
	By Proposition \ref{maindelta}, for $i=1,\ldots,d-1$, we have 
	$$(i+1)(\delta_i(\Pc)-\delta_i(\Qc))+(d-i+1)(\delta_{i-1}(\Pc)-\delta_{i-1}(\Qc))=0$$
	Since $\delta_0(\Pc)=\delta_0(\Qc)$, for $i=0,\ldots,d-1$, we have $\delta_i(\Pc)=\delta_i(\Qc)$.
	Hence $\delta(\Pc)=\delta(\Qc)$.
\end{proof}
 
 We let $\Pc \subset \RR^2$ be a reflexive polytope of dimension $2$. 
Then the $\delta$-vector of $\Pc$ equals the $\delta$-vector of $\Pc^\vee$ if and only if $\Pc \cong \Pc^\vee$.
However, there exists a reflexive polytope of dimension 3 whose $\delta$-vector equals the $\delta$-vector of the dual polytope such that it and the dual polytope are not unimodularly equivalent.
We give an example of such a reflexive polytope.

\begin{example}[{\cite[Example 35.11]{Hibi}}]
	\label{eq_nue}
	Let $\Pc \subset \RR^3$ be the reflexive polytope with the vertices $(-1,0,1)$, $(-1,0,-1)$, $(1,1,1)$, $(1,1,-1)$, $(0,-1,1)$ and $(0,-1,-1)$.
	Then $\Pc$ has 5 facets.
	Hence $\Pc^\vee$ has 5 vertices $(0,0,1)$, $(0,0,-1)$, $(2,-1,0)$, $(-1,2,0)$ and $(-1,-1,0)$.
	Therefore $\Pc$ and $\Pc^\vee$ are not unimodularly equivalent.
	However, $\delta(\Pc)=\delta(\Pc^\vee)=(1,8,8,1)$.
\end{example}

By using Theorem \ref{main} and Example \ref{eq_nue}, we obtain the following corollary.

\begin{corollary}
	For each $d \geq 3$, there exists a reflexive polytope of dimension $d$ such that $\delta(\Pc)=\delta(\Pc^\vee)$ but $\Pc$ and $\Pc^\vee$  are not unimodularly equivalent.
\end{corollary}

\section{a new example of a reflexive simplex}
For $d \geq 2$, we  let $\Pc$ be a reflexive polytope of dimension $d$ and $\Qc$ the reflexive polytope of dimension $d+1$ satisfying the assumption of Theorem \ref{main}.
Clearly, $\Qc$ is not simplex.
We consider a reflexive simplex whose $\delta$-vector equals the $\delta$-vector of the dual polytope.

First, we give  an elementary number-theoretic notion.
\begin{definition}
	The well-known recursive sequence (\cite[A000058]{Slo}) of pairwise coprime natural numbers $b_0:=2$, $b_n:=1+b_0\cdots b_{n-1} (n \geq 1)$ is called \textit{Sylvester Sequence}.
	It starts as $b_0=2$, $b_1=3$, $b_2=7$, $b_3=43$, $b_4=1807$.
\end{definition}

For $d \geq 2$, we  let $\Pc$ be a reflexive simplex of dimension $d$.
It is known  
$$(d+1)^{d+1} \leq \text{Vol}(\Pc)\text{Vol}(\Pc^\vee)  \leq (b_d-1)^2,$$
and if $\text{Vol}(\Pc)=b_d-1$, then $\Pc  \cong \Pc^\vee$ (\cite[Theorem C]{Nill}).
Hence if $\Pc \cong \Pc^\vee$, we have $\text{Vol}(\Pc) \leq b_d-1$.

In this section for $d \geq 3$ we give a new example of a reflexive simplex  of dimension $d$ whose $\delta$-vector equals the $\delta$-vector of the dual polytope.
In particular the simplex and the dual polytope are unimodularly equivalent and the volume is less than $b_d-1$.
In fact,
	\begin{theorem}
	\label{sim}
	For $d \geq 3$, let $\Pc$ be the $d$-dimensional simplex whose vertices $v_i \in \RR^d,i=0,1,\ldots,d$, are of the form:
	\begin{displaymath}
	v_i=\left\{        
	\begin{aligned} 
	&-3e_1-2\sum^{d}_{i=2}e_i  &i&=0,\\
	&e_1 &i&=1,\\
	&e_1+2e_i &i&=2,3,\\
	&e_1+2b_{i-4}e_i &i&=4,\ldots ,d,
	\end{aligned}
	\right.
	\end{displaymath}
	where $e_i$ denotes the $i$-th unit vector.
	Then $\Pc$ is reflexive and we have $\Pc \cong \Pc^\vee$, in particular, $\textnormal{Vol}(\Pc)<b_d-1$.
	
\end{theorem}

In order to prove Theorem \ref{sim}, we use the following lemma. 
\begin{lemma} [{\cite[A000058]{Slo}}]
	\label{b_frac}
	For each $n \geq 0$
	$$\cfrac{1}{b_0}+\cfrac{1}{b_1}+\cdots +\cfrac{1}{b_n}=1-\cfrac{1}{b_0\cdots b_n}.$$
\end{lemma}

Now, we prove Thorem \ref{sim}.
\begin{proof}[Proof of Theorem \ref{sim}]
	First, we show that $\Pc$ is reflexive.
	Let $\Fc_0,\ldots ,\Fc_d$ be facets of $\Pc$, which are of the form:
	$$\Fc_i=\conv(\left\{v_0,\ldots,v_{i-1},v_{i+1},\ldots,v_d \right\}) \ \ \ \ \ 0 \leq i \leq d,$$
	and for $0 \leq i \leq d$, let $\Hc_i$ be a hyperplane  satisfying $\Fc_i=\Pc \cap \Hc_i$.
	Then  
	\begin{displaymath}
	\Hc_i=\left\{        
	\begin{aligned} 
	&\left\{(x_1,\ldots ,x_d) \in \RR ^d  | x_1=1 \right\} &i&=0,\\
	&\left\{(x_1,\ldots ,x_d) \in \RR ^d  | x_1-2x_i=1 \right\} &i&=2,\ldots ,d.
	\end{aligned}
	\right.
	\end{displaymath}
	Also $\Hc_1=\left\{(x_1,\ldots ,x_d) \in \RR ^d  |  \sum^{d}_{i=1}a_ix_i=1 \right\}$, where
	\begin{displaymath}
	a_i=\left\{        
	\begin{aligned} 
	&-(4b_0\cdots b_{d-4}-1)  &i&=1,\\
	&\cfrac{4b_0\cdots b_{d-4}}{2} &i&=2,3,\\
	&\cfrac{4b_0\cdots  b_{d-4}}{2b_{i-4}} &i&=4,\ldots d.
	\end{aligned}
	\right.
	\end{displaymath}
	In fact, $v_0\in \Hc_1$ since
	\begin{displaymath}
	\begin{aligned}
	&3(4b_0 \cdots b_{d-4}-1)-4b_0 \cdots b_{d-4}-4b_0 \cdots b_{d-4}-\cfrac{4b_0 \cdots b_{d-4}}{b_0}-\cdots -\cfrac{4b_0 \cdots b_{d-4}}{b_{d-4}}\\
	=&-3+4b_0 \cdots b_{d-4}(1-(\cfrac{1}{b_0}+\cfrac{1}{b_1}+\cdots +\cfrac{1}{b_{d-4}}))\\
	=&-3+4b_0\cdots b_{d-4}\cfrac{1}{b_0 \cdots b_{d-4}}\ \ \ \  ({\rm Lemma} \ \ref{b_frac})\\
	=&1.
	\end{aligned}
	\end{displaymath}
	Hence since $a_i \in \ZZ \ (1 \leq i \leq d)$, by Lemma \ref{facet}, $\Pc$ is reflexive.
	
	Next, we show that $\Pc \cong \Pc^\vee$.
	By Lemma \ref{facet}, we obtain that $w_0,\ldots,w_d$ are the vertices of $\Pc^\vee$, where
	\begin{displaymath}
	w_i =\left\{        
	\begin{aligned} 
	&e_1 &i&=0,\\
	&(a_1,\ldots ,a_d) &i&=1,\\
	&e_1-2e_i &i&=2,\ldots ,d
	\end{aligned}
	\right.
	\end{displaymath}
	We set a $d \times d$ matrix 
	\begin{displaymath}
	U=\left(  
	\begin{array}{cccccc}
	1 & 2 & \  & \   & \  & \ \\
	2 & 2   & 1  & 1 &  \cdots & 1 \\
	\  & 1 & -1 &  \  & \  & \ \\
	\  & 1 & \  & -b_0 &  \ & \ \\
	\  & \vdots & \   & \  & \ddots & \  \\
	\  & 1 & \  & \   & \  & -b_{d-4}
	\end{array}
	\right)
	\in \ZZ^{d \times d},
	\end{displaymath}
	where all other terms are zero.
	Then by Lemma \ref{b_frac}, we have
	\begin{displaymath}
	\begin{aligned}
	\text{det}(U)&=
	\text{det}\left(
	\begin{array}{ccc}
	1 & 2 & 0  \\
	2 & 2  & 1 \\
	0 & 1 & -1  
	\end{array}
	\right)
	(-1)^{d-3}b_0 \cdots b_{d-4}+(-1)^{d-2} \sum ^{d-4}_{i=0}\cfrac{b_0 \cdots b_{d-4}}{b_i}\\
	&=(-1)^{d-3}b_0\cdots b_{d-4}+(-1)^{d-2} \sum ^{d-4}_{i=0}\cfrac{b_0 \cdots b_{d-4}}{b_i}\\
	&=(-1)^{d-3}b_0\cdots b_{d-4}(1-\sum ^{d-4}_{i=0}\cfrac{1}{b_i})\\
	&=(-1)^{d-3}.
	\end{aligned}
	\end{displaymath}
	Hence $U$ is a unimodular matrix.
	Let $f_U$ be the linear transformation in $\RR^d$ defined by $U$.
	Then 
	\begin{displaymath}
	f_U(w_i)=\left\{        
	\begin{aligned} 
	&v_2 &i&=0,\\
	&v_1 &i&=1,\\
	&v_0 &i&=2,\\
	&v_i &i&=3,\ldots ,d.
	\end{aligned}
	\right.
	\end{displaymath}
	Hence $\Pc=f_U(\Pc^\vee)$. 
	Therefore we have $\Pc \cong\Pc^\vee$.
	
	Finally, we show that $\textnormal{Vol}(\Pc)<b_d-1$.
	If $d=3$, then  $\textnormal{Vol}(\Pc)=16<42=b_3-1$.
	We assume that $d \geq 4$.
	Since for each $n \geq 1$, $b_n>b_0=2$, for each $n \geq 0$, we have $b_n>2^n$. 
	Hence since $d \geq 4$ and since $\text{Vol}(\Pc)=|2^{d-1}(a_1-1)|=2^{d+1}b_0\cdots b_{d-4}$,
	we have 
	$$b_d-1=b_0\cdots b_{d-1}	>2^{3d-6}b_0\cdots b_{d-4}>2^{d+1}b_0\cdots b_{d-4}=\text{Vol}(\Pc),$$
	as desired.
\end{proof}

\begin{acknowledgement}
	The author would like to thank Professor Takayuki Hibi and Professor Akihiro Higashitani for helping me in writing this paper,
	and to thank anonymous referees for reading the manuscript carefully.
\end{acknowledgement}

\end{document}